\documentclass[12pt,british,a4paper,reqno]{amsart}

	\usepackage[T1]{fontenc}
	\usepackage[latin9]{inputenc}
	\usepackage[british]{babel}
	\DeclareMathSizes{12}{12}{7}{6}

	\usepackage{amsmath}
	\usepackage{amssymb}
	\usepackage{amsthm}
	\usepackage{bbm}
	\usepackage{braket}
	\usepackage{xcolor}
	\usepackage[shortlabels]{enumitem}
	\usepackage{fancyhdr}
	\usepackage{graphicx}
	\usepackage{ifsym}
	\usepackage{indentfirst}
	\usepackage{lmodern}
	\usepackage{mathrsfs}
	\usepackage{mathtools}
	\usepackage{oplotsymbl} 
	\usepackage{parskip}
	\usepackage{proof}
	\usepackage{qtree}
	\usepackage{scalerel}
	\usepackage{setspace}
	\usepackage{stmaryrd}
	\usepackage{tensor}
	\usepackage{tikz}
	\usepackage{tikz-cd}
	\usepackage{wasysym}
	\usepackage{xparse}

	\usepackage[
	hidelinks
		]{hyperref}
		\usepackage[capitalize]{cleveref}
		\usepackage{bookmark}
	\usepackage{geometry}

	\bookmarksetup{numbered,open}

	\renewcommand{\geq}{\geqslant}
	\renewcommand{\leq}{\leqslant}
	\renewcommand{\phi}{\varphi}

	\providecommand{\corollaryname}{Corollary}
	\providecommand{\definitionname}{Definition}
	\providecommand{\examplename}{Example}
	\providecommand{\lemmaname}{Lemma}
	\providecommand{\notationname}{Notation}
	\providecommand{\propositionname}{Proposition}
	\providecommand{\remarkname}{Remark}
	\providecommand{\theoremname}{Theorem}
	\providecommand{\setupname}{Setup}
	\providecommand{\conjecturename}{Conjecture}
	\providecommand{\questionname}{Question}
	\providecommand{\claimname}{Claim}

	\theoremstyle{plain}
		\newtheorem{thm}{\protect\theoremname}[section] 
		
		\newtheorem{prop}[thm]{\protect\propositionname}
		\newtheorem{lem}[thm]{\protect\lemmaname}
		\newtheorem{cor}[thm]{\protect\corollaryname}

	\theoremstyle{definition}
		\newtheorem{defn}[thm]{\protect\definitionname}

		\newtheorem{example}[thm]{\protect\examplename}
		\newtheorem{setup}[thm]{\setupname}

	\theoremstyle{remark}
		\newtheorem{rem}[thm]{\protect\remarkname}
		
	\numberwithin{figure}{section}
	\numberwithin{equation}{section}

	\usetikzlibrary{matrix,arrows,decorations.pathmorphing,positioning,decorations.pathreplacing}
	\tikzset{commutative diagrams/.cd, 
		mysymbol/.style = {start anchor=center, end anchor = center, draw = none}}
	\tikzcdset{every label/.append style = {font = \footnotesize}}

	\newcommand{\commutes}[2][\circlearrowleft]{\arrow[mysymbol]{#2}[description]{#1}}

	\newcommand{\BZ}{\mathbb{Z}}

	\newcommand{\CA}{\mathcal{A}}
	\newcommand{\CB}{\mathcal{B}}

		\newcommand{\ring}{k}

		\newcommand{\sse}{\subseteq}

		\newcommand{\binprod}{\mathbin\Pi}
		\newcommand{\bincoprod}{\mathbin\amalg}
		\newcommand{\Ker}{\operatorname{Ker}\nolimits}
		\newcommand{\Cok}{\operatorname{Coker}\nolimits}

		\newcommand{\iso}{\cong}

		\newcommand{\Hom}{\operatorname{Hom}\nolimits}
		\newcommand{\End}{\operatorname{End}\nolimits}

		\newcommand{\into}{\hookrightarrow}
		\newcommand{\onto}{\rightarrow\mathrel{\mkern-14mu}\rightarrow}
		
		\newcommand{\cok}{\operatorname{coker}\nolimits}

		\newcommand{\id}[1]
								{\mathrm{id}{_{#1}}}

		\newcommand{\rMod}[1]{\operatorname{\mathsf{Mod}}\nolimits{#1}}

		\newcommand{\rproj}[1]{\operatorname{\mathsf{proj}}\nolimits{#1}}

		\newcommand{\add}[1]{\operatorname{\mathsf{add}}\nolimits{#1}}

	\newcommand{\sym}{\operatorname{Sym}\nolimits} 

	\newcommand{\deff}{\coloneqq}
	
	\newcommand\restr[2]{{\left.\kern-\nulldelimiterspace#1
						\right|_{#2}}}

\makeatletter

	\renewcommand{\andify}{%
		\nxandlist{\unskip, }{\unskip{} \@@and~}{\unskip \penalty-2 \space \@@and~}}
    
	\renewcommand\author@andify{%
  		\nxandlist {\unskip ,\penalty-1 \space\ignorespaces}%
		{\unskip {} \@@and~}%
		{\unskip \penalty-2 \space \@@and~}
	}

	    \newenvironment{acknowledgements}{%
	    \renewcommand\abstractname{Acknowledgements}
\global\setbox\abstractbox=\vtop \bgroup
\normalfont\Small
\list{}{\labelwidth\z@
\leftmargin3pc \rightmargin\leftmargin
\listparindent\normalparindent \itemindent\z@
\parsep\z@ \@plus\p@

}%
\item[\hskip\labelsep\scshape\abstractname.]%
}{%
\endlist\egroup
\ifx\@setabstract\relax \@setabstracta \fi
}

\def\@setaddresses{\par
    \nobreak \begingroup
    \setstretch{0.5} 
    \footnotesize
    \def\author##1{\nobreak\addvspace\bigskipamount}%
    \def\\{\unskip, \ignorespaces}%
    \interlinepenalty\@M
    \def\address##1##2{\begingroup
        \par\addvspace\bigskipamount\indent
    \@ifnotempty{##1}{(\ignorespaces##1\unskip) }%
    {\scshape\ignorespaces##2}\par\endgroup}%
    \def\curraddr##1##2{\begingroup
    \@ifnotempty{##2}{\nobreak\indent\curraddrname
    \@ifnotempty{##1}{, \ignorespaces##1\unskip}\/:\space
    ##2\par}\endgroup}%
    \def\email##1##2{\begingroup
    \@ifnotempty{##2}{\nobreak\indent\emailaddrname
    \@ifnotempty{##1}{, \ignorespaces##1\unskip}\/:\space
    \ttfamily##2\par}\endgroup}%
    \def\urladdr##1##2{\begingroup
    \def~{\char'\~}%
    \@ifnotempty{##2}{\nobreak\indent\urladdrname
    \@ifnotempty{##1}{, \ignorespaces##1\unskip}\/:\space
    \ttfamily##2\par}\endgroup}%
    \addresses
    \endgroup
}

\let\oldtocsection=\tocsection
\let\oldtocsubsection=\tocsubsection

\renewcommand{\tocsection}[2]{\vspace*{0pt}\hspace{0em}\oldtocsection{#1}{#2}}

    \renewcommand{\tocsubsection}[2]{\vspace*{0pt}\hspace{21pt}\oldtocsubsection{#1}{#2}}

\let\amph\&
\makeatother

\begin{document}

\title{Krull-Remak-Schmidt decompositions in {H}om-finite additive categories}
\author{Amit Shah}
	\address{
		Department of Mathematics\\
		Aarhus University\\
		8000 Aarhus C\\
		Denmark
	}
    \email{amit.shah@math.au.dk}
\date{\today}
\dedicatory{Dedicated to Robert E.\ Remak (1888--1942) in honour of his contributions to mathematics.}
%
{\setstretch{1}\begin{abstract}
An additive category in which each object has a Krull-Remak-Schmidt decomposition---that is, a finite direct sum decomposition consisting of objects with local endomorphism rings---is known as a Krull-Schmidt category. A $\Hom$-finite category is an additive category $\mathcal{A}$ for which there is a commutative unital ring $k$, such that each $\Hom$-set in $\mathcal{A}$ is a finite length $k$-module. The aim of this note is to provide a proof that a $\Hom$-finite category is Krull-Schmidt, if and only if it has split idempotents, if and only if each indecomposable object has a local endomorphism ring. 
\end{abstract}}
\maketitle

\section{Introduction}
\label{sec:introduction}

If one blurts out ``decomposition theorem'' to an undergraduate in mathematics, one might expect them to think of the fundamental theorem of arithmetic or perhaps the fundamental theorem of finitely generated abelian groups. Such results are of interest (and importance) because we can hope to understand a more complicated object by first understanding the simpler components of which it is comprised. 
It is this kind of application that has made another famous decomposition theorem of such wide interest. 
Let $\CA$ be an additive category. 
An object $X$ is said to satisfy the \emph{Krull-Remak-Schmidt theorem} if, whenever $M_{1} \oplus \cdots \oplus M_{m}$ and $N_{1} \oplus \cdots \oplus N_{n}$ are finite direct sum decompositions of $X$ into objects each having local endomorphism rings, then $m=n$ and there is a permutation $\sigma$ in $\sym(n)$ such that $M_{j}$ is isomorphic to $N_{\sigma(j)}$ for $j=1,\ldots, n$. 

The Krull-Remak-Schmidt theorem has its roots in finite group theory: 
first Frobenius--Stickel\-berger \cite{FrobeniusStickelberger-Uber-gruppen-von-vertauschbaren-elementen} demonstrated it for finite abelian groups; 
Remak \cite{Remak-Uber-die-zerlegung-der-endlichen-gruppen-in-direkte-unzerlegbare-faktoren} for finite groups\footnote{Remak proved a stronger conclusion: $m=n$, and there exists $\sigma\in\sym(n)$ and an automorphism $f\colon X\to X$ that is the identity on $X$ modulo its centre, such that $M_{j}\iso N_{\sigma(j)}$ via $f$ for $1\leq j \leq n$. Wedderburn \cite{Maclagan-Wedderburn-On-the-direct-product-in-the-theory-of-finite-groups} proposed a proof for finite groups slightly before Remak, but with no central isomorphism aspect. Moreover, it is not clear if Wedderburn's proof is complete, with Remak having commented on deficiencies in the proof at the end of \cite{Remak-Uber-die-zerlegung-der-endlichen-gruppen-in-direkte-unzerlegbare-faktoren}.};  
Schmidt \cite{Schmidt-Sur-les-produits-directs} also for finite groups but with a  substantially shorter proof; 
and 
Krull \cite{Krull-uber-verallgemeinerte-endliche-abelsche-gruppen} for abelian operator groups (usually stated in the language of modules) with ascending and descending chain conditions. 
The first categorical version of the theorem was established by Atiyah \cite{Atiyah-KS-theorem-with-apps-to-sheaves}. 
For a nice introduction on the Krull-Remak-Schmidt theorem for module categories see 
Facchini \cite{Facchini-the-KS-theorem}, 
and for additive categories see 
Walker--Warfield \cite{WalkerWarfield-Unique-decomposition-and-isomorphic-refinement-theorems-in-additive-categories}. 

A finite direct sum decomposition of an object in $\CA$ into objects having local endomorphism rings is known as a \emph{Krull-Remak-Schmidt decomposition}. 
If each object in $\CA$ admits such a decomposition, then $\CA$ is known as a \emph{Krull-Schmidt category}. 
If there is a commutative unital ring $\ring$ such that each $\Hom$-set in $\CA$ is a finite length $\ring$-module, then we call $\CA$ \emph{$\Hom$-finite}. 
The purpose of this note is to show that if $\CA$ is $\Hom$-finite, 
then $\CA$ is Krull-Schmidt, if and only if it is idempotent complete, if and only if the endomorphism ring of any indecomposable object in $\CA$ is local (see \cref{thm:main-theorem}). 
The equivalence of the first two conditions follows from 
the theory of projective covers; see Chen--Ye--Zhang \cite[\S A.1]{ChenYeZhang-Algebras-of-derived-dimension-zero}, 
Krause \cite[\S 4]{Krause-KS-cats-and-projective-covers}, or 
\cref{cor:Kr15-Cor-4-4-KS-category-iff-split-idems-and-semi-perfect-endo-rings}. 
The motivation for this note is the equivalence of the latter two conditions. 
Although this is certainly known when $k$ is a field 
(see e.g.\ 
Ringel \cite[\S 2.2]{Ringel-tame-algebras-and-integral-quadratic-forms}, 
Happel \cite[\S I.3.2]{Happel-triangulated-cats-in-rep-theory}, 
Chen--Ye--Zhang \cite[Cor.\ A.2]{ChenYeZhang-Algebras-of-derived-dimension-zero}), 
it remains true when $k$ is any commutative unital ring. 
However, the author failed to find a proof in the literature. 

We assume the reader is familiar with the theory of modules and the notions of a category and a functor. In Sections~\ref{sec:Krull-Schmidt-categories-semi-perfect-rings} and \ref{sec:subobjects-bichain-Hom-finite} we rely on several results from \cite{Krause-KS-cats-and-projective-covers}, which we typically do not reprove here. As such, this note is not self-contained. However, by keeping \cite{Krause-KS-cats-and-projective-covers} to hand the reader should not struggle. 
	
This article is organised as follows. 
	We recall some concepts from category theory in \cref{sec:categories}. 
	\cref{sec:idempotents} contains some preliminaries on idempotents and local rings. 
	In \cref{sec:Krull-Schmidt-categories-semi-perfect-rings} we turn to Krull-Schmidt categories and the relation to semi-perfect rings. 
	We recall some material on bi-chain conditions in abelian categories in \cref{sec:subobjects-bichain-Hom-finite}. 
	Lastly, we demonstrate the main result---\cref{thm:main-theorem}---in \cref{sec:main-theorem}.

\begin{rem}
It is not so clear why the terminology `Krull-Schmidt category' fails to include Remak's name. 
It perhaps originates from 
D\"{u}r \cite{Dur-Mobius-functions-incidence-algebras-and-power-series-representations}, in which D\"{u}r mentions \cite{Atiyah-KS-theorem-with-apps-to-sheaves} for influence on the choice of terminology. 
Atiyah referred only to the `Krull-Schmidt theorem' and to a `Remak decomposition'. 
Although this terminology is now entrenched in our mathematical language, we ought to remember that the statement Remak proved for finite groups is the assertion Schmidt proved in a shorter way and then Krull generalised. 

Remak demonstrated this significant result in his PhD thesis in 1911. 
His contributions to mathematics covered a wide range of areas, including group theory, number theory and analysis. 
Remak was murdered in Auschwitz in or after 1942 \cite[p.\ 64]{Segal-Mathematicians-under-the-nazis}. 
A nice biography of Remak is given by Merzbach \cite{Merzbach-Robert-Remak-and-the-estimation-of-units-and-regulators}. 
\end{rem}


\section{Additive and abelian categories}
\label{sec:categories}

The notions we recall here are standard. We mainly use this section to set up notation for the remainder of the article. An accessible introduction to these concepts can be found in Aluffi \cite[Chs.\ I, IX]{Aluffi-Chapter0}. The reader who is more familiar with category theory can safely skip this section.

Suppose \(\CA\) is a category and let $X,Y$ be objects in $\CA$.
We denote the collection of morphisms \(X\to Y\) in \(\CA\) by \(\Hom_{\CA}(X,Y)\). 
The collection of \emph{endomorphisms $f\colon X\to X$} of $X$ are denoted \(\End_{\CA}(X)\).
A \emph{zero object} in $\CA$ is an object $X\in\CA$ for which $\Hom_{\CA}(X,Y)$ and $\Hom_{\CA}(Y,X)$ are both singletons for each $Y\in\CA$.

\begin{defn}
For $X,Y\in\CA$, a \emph{coproduct} of $X$ and $Y$ is an object
$X\bincoprod Y$ in $\CA$ endowed with morphisms $i_{X}\colon X\to X\bincoprod Y$
and $i_{Y}\colon Y\to X\bincoprod Y$ satisfying the following universal
property: given $Z\in\CA$ and morphisms $f_{X}\colon X\to Z$, $f_{Y}\colon Y\to Z$,
there exists a unique morphism $g \colon X\bincoprod Y\to Z$ such that the diagram below commutes in $\CA$.
\[
\begin{tikzcd}
X \arrow{dr}[swap]{i_{X}} \arrow[bend left]{drr}{f_{X}} \commutes{drr}& & \\
 & X \bincoprod Y \arrow[dotted]{r}{\exists ! g} & Z \\
Y \arrow{ur}{i_{Y}} \arrow[bend right]{urr}[swap]{f_{Y}} \commutes{urr} & &
\end{tikzcd}
\]

A \emph{product} $X\binprod Y$ of $X$ and $Y$ is the dual notion, and we omit the description here.
\end{defn}

As with all objects defined by a universal property of this kind, (co)products and zero objects (where they exist) are unique up to unique isomorphism. 
Furthermore, although we only defined binary (co)products above, one can define finite (co)products similarly. 
We are now in a position to define an additive category.

\begin{defn}
\label{def:additive cat}
The category $\CA$ is called \emph{preadditive} if 
\begin{enumerate}[label=\textup{(\roman*)}]
	\item $\Hom_{\CA}(X,Y)$ has the structure of an abelian group for all $X,Y$ in $\CA$, such that the composition function $\Hom_{\CA}(X,Y) \times \Hom_{\CA}(Y,Z) \to \Hom_{\CA}(X,Z)$ 
	(sending $(f,g)$ to $gf = g\circ f$) 
	is $\BZ$-bilinear for all objects $X,Y,Z$ in $\CA$. 
\end{enumerate}
The category $\CA$ is \emph{additive} if it is preadditive and it has 
\begin{enumerate}[label=\textup{(\roman*)}]
\setcounter{enumi}{1}
	\item a zero object, which we denote by \(0\), and 
	\item finite products and finite coproducts. 
\end{enumerate}
\end{defn}

For the rest of this section suppose that \(\CA\) is an additive category.

\begin{rem}
\label{rem:on-the-definition-of-additive-category}
Let \(X,Y\in\CA\) be objects. 
\begin{enumerate}[label=\textup{(\roman*)}]	
	\item We denote the abelian group operation of $\Hom_{\CA}(X,Y)$ by $+$, and the identity element by $0$. 
		
	\item The product $X\binprod Y$ and coproduct $X\bincoprod Y$ are isomorphic;  
	see Mac Lane \cite[Exer.\ VIII.2.1]{MacLane-categories-for-the-working-mathematician}. 
	This object 
	is denoted $X\oplus Y$ and called the \emph{direct sum} of \(X\) and \(Y\). 
	In particular, it is equipped with morphisms
	\(i_{X}\colon X \to X\oplus Y\), 
	\(i_{Y}\colon Y \to X\oplus Y\), 
	\(p_{X}\colon X\oplus Y \to X\) and 
	\(p_{Y}\colon X\oplus Y \to Y\), 
	such that 
	\(p_{X}i_{X} = \id{X}\),
	\(p_{Y}i_{Y} = \id{Y}\) and 
	\(i_{X}p_{X} + i_{Y}p_{Y} = \id{X\oplus Y}\). 
	These equations also imply that \(p_{X}i_{Y} = 0\) and \(p_{Y}i_{X} = 0\). 

\end{enumerate}
\end{rem}

Later we deal with additive categories that have extra structure on their $\Hom$-sets. 
The following notion captures this. 
We always assume rings are associative and unital. 
However, we do not necessarily assume the additive identity \(0\) and the multiplicative identity \(1\) of a ring are distinct.

\begin{defn}
\label{def:R-linear-category}
Let $\ring$ be a commutative ring. 
The category $\CA$ is called a \emph{$\ring$-linear} category if 
$\Hom_{\CA}(X,Y)$ is a $\ring$-module for all objects $X,Y$ in $\CA$ and the composition of morphisms is $\ring$-bilinear.
\end{defn}

\begin{example}
\label{example:R-linear-categories}
\begin{enumerate}[label=\textup{(\roman*)}]
	\item 
	Recall that an abelian group is nothing other than a $\BZ$-module. 
	Thus, a category that is additive in the sense of \cref{def:additive cat} is just a $\BZ$-linear category in the sense of \cref{def:R-linear-category}.
	
	\item\label{item:R-Mod-is-R-linear}
	 Let $\Lambda$ be a ring. 
	We denote by $\rMod{\Lambda}$ the category of all right $\Lambda$-modules. 
	If \(\Lambda = \ring\) is commutative, then $\rMod{\ring}$ is $\ring$-linear. 
	See e.g.\ \cite[Chp.\ III]{Aluffi-Chapter0} for details. 
\end{enumerate}
\end{example}

Lastly we recall the concepts of (co)kernels and abelian categories.

\begin{defn}
\label{def:weak-co-kernel}
Let \(f\colon X\to Y\) be a morphism in \(\CA\).
A \emph{weak kernel} of \(f\) is a morphism \(i\colon K\to X\)
in \(\CA\) with \(fi=0\), and such that: 
for any \(a\colon A\to X\) with \(fa=0\) there
exists \(b\colon A\to K\) 
such that
\(a=ib\). 
\[
\begin{tikzcd}
	& A \arrow{d}{a}\arrow[dotted]{dl}[swap]{\exists b}& \\
K \arrow{r}{i}& X \arrow{r}{f}& Y
\end{tikzcd}
\]
If the morphism \(b\) obtained in this way is always uniquely determined, 
then \(i\colon K \to X\) is called a \emph{kernel} of \(f\). 

One defines a \emph{weak cokernel} and a \emph{cokernel} of \(f\) dually. 
\end{defn}

Recall that a morphism \(f\colon X\to Y\) in the additive category \(\CA\) is \emph{monic} (or a \emph{monomorphism}) 
if for any morphism \(a\colon A \to X\) with \(fa=0\) we must have that \(a=0\). 
An \emph{epic} morphism (or \emph{epimorphism}) is defined dually.

\begin{defn}
\label{def:abelian-category}
An additive category $\CA$ is said to be \emph{abelian} if: 
\begin{enumerate}[(i)]
\item each morphism \(f\colon X\to Y\) in $\CA$ has a kernel \(\ker f\colon \Ker f\to X \) and a cokernel \(\cok f\colon Y \to \Cok f\); and
\item in $\CA$ every monomorphism is the kernel of some morphism, and every epimorphism is the cokernel of some morphism.
\end{enumerate}
\end{defn}

\begin{example}
\label{example:R-Mod-is-R-linear-abelian}
If $\ring$ is a commutative ring, then the category $\rMod{\ring}$ 
(see \cref{example:R-linear-categories}\ref{item:R-Mod-is-R-linear}) 
is an abelian \(\ring\)-linear category.
\end{example}

We conclude this section with the following straightforward lemma.

\begin{lem}
\label{lem:monic-weak-kernel-is-kernel-and-dual}
A morphism $f$ in $\CA$ is a kernel if and only if it is a monic weak kernel. 
It is a cokernel if and only if it is an epic weak cokernel.
\end{lem}


\section{Idempotents}
\label{sec:idempotents}

We will see in the next section that an additive category having Krull-Remak-Schmidt decompositions is very closely related to it having so-called split idempotents. 
However, more generally, one can define idempotents in any ring and we begin with such considerations now.

\subsection{Idempotents in rings}

Let \(\Lambda\) be a ring with multiplicative identity $1 = 1_{\Lambda}$.

\begin{defn} 
\begin{enumerate}[label=\textup{(\roman*)}]
	\item An element \(e\in\Lambda\) is called an \emph{idempotent} if \(e^{2}=e\). 

	\item Two idempotents \(e,f\in\Lambda\) are called \emph{orthogonal} if \(ef=0=fe\). A set \(\{e_{j}\}_{j=1}^{n}\sse \Lambda\) of idempotents is called  \emph{orthogonal} if 
	its elements are pairwise orthogonal. 
	
	\item An idempotent \(e\in\Lambda\) is called \emph{primitive} if 
	\(e=f+g\) implies \(f=0\) or \(g=0\) 
	for all orthogonal idempotents \(f,g\in\Lambda\).
	
	\item A set \(\{e_{j}\}_{j=1}^{n}\sse \Lambda\) of idempotents is called \emph{complete} if \(e_{1}+\cdots+e_{n}=1\). 
	
\end{enumerate}
\end{defn}

\begin{example}
\label{example:idempotents}
\begin{enumerate}[label=\textup{(\roman*)}]
	\item Let \(\CA\) be an additive category and suppose \(X\in\CA\) is an object. 
	The identity morphism \(\id{X}\colon X\to X\) and the zero morphism \(0\colon X\to X\) are always idempotents of the endomorphism ring 
	\(
	\End_{\CA}(X)
			= \Hom_{\CA}(X,X)
	\). 
	
	\item\label{item:1-e-is-idempotent} Given any idempotent \(e\in\Lambda\), 
	the element \(1-e\in\Lambda\) is also idempotent, and 
	\(e\) and \(1-e\) are orthogonal. 
	Furthermore, the right $\Lambda$-module $\Lambda_{\Lambda}$ decomposes as $\Lambda = e\Lambda \oplus (1-e)\Lambda$. 
\end{enumerate}
\end{example}

Lastly in this subsection, we recall the definition of a local ring, and its connection with idempotents and indecomposable modules.

\begin{defn}
\label{def:local-ring}
If \(0\neq 1\) and the sum of any two non-units in $\Lambda$ is again a non-unit, then \(\Lambda\) is called a \emph{local} ring. 
\end{defn}

\begin{rem}
\label{rem:local-ring}
A ring \(\Lambda\) is local if, equivalently, \(0\neq 1\) and \(\Lambda\) has a unique maximal right ideal. 
This right ideal is precisely the collection of non-units in such a ring. 
In particular, given an element \(x\) in a local ring \(\Lambda\), we have that \(x\) or \(1-x\) is invertible. See Anderson--Fuller \cite[Prop.\ 15.15]{AndersonFuller-rings-and-cats-of-modules}, or Lam \cite[Thm.\ 19.1]{Lam-first-course-noncomm-rings}.
\end{rem}

Local rings have very few idempotents as we now see. 
The converse of the following lemma holds if \(\Lambda\) is artinian; 
see e.g.\ \cite[Cor.\ 19.19]{Lam-first-course-noncomm-rings}.

\begin{lem}
\label{lem:local-ring-has-0-1-idempotents-only}
If \(\Lambda\) is local,  
then \(\Lambda\) has precisely two idempotents \(0\) and \(1\). In particular, the idempotent $1$ is primitive. 
\end{lem}

\begin{proof}
Let \(e\in\Lambda\) be an idempotent. 
Then \(1-e\) is also idempotent by \cref{example:idempotents}\ref{item:1-e-is-idempotent}. 
Since \(\Lambda\) is local, we have that \(e\) or \(1-e\) is invertible by \cref{rem:local-ring}. 
If \(e\) is a unit and \(ef=1=fe\) for some \(f\in\Lambda\), then \(e=e\cdot1=e(ef)=ef=1\)
as \(e=e^{2}\). 
On the other hand, if \(1-e\) is invertible then 
a similar argument shows that 
\(1-e=1\), 
whence \(e=0\). 

For the other assertion, suppose $1 = e+f$ for some orthogonal idempotents $e,f\in\Lambda$. If $e = f= 0$ then $1= 0$, which is impossible. Therefore, without loss of generality, $e = 1$ and so $f = 0$. 
\end{proof}

A \emph{corner ring} of \(\Lambda\) is a ring of the form \(e\Lambda e\) for some idempotent \(e\in\Lambda\). 
Note that \(e\Lambda e\) is unital with unit \(1_{e\Lambda e} = e\). 
There is an isomorphism 
\begin{equation}
\label{eqn:isomorphism-Phi}
\Phi\colon \End_{\rMod{\Lambda}}(e\Lambda)
				\overset{\iso}{\longrightarrow} e\Lambda e
\end{equation}
of right \(e \Lambda e\)-modules 
given by 
\(
\Phi(h) \deff h(e) = h(e)e
\). 
Moreover, $\Phi$ is an isomorphism of rings. 
See 
Assem--Simson--Skowro\'{n}ski \cite[Lem.\ I.4.2]{AssemSimsonSkowronski-Vol1}.

A non-zero module \(M\) is called \emph{indecomposable}
if, whenever there is an isomorphism \(M\iso M_{1}\oplus M_{2}\) of modules, we have \(M_{1}=0\) or \(M_{2}=0\). 
When we have a decomposition like \(M\iso M_{1}\oplus M_{2}\), 
we usually more simply write \(M = M_{1}\oplus M_{2}\). 
Nothing is lost with this identification for our purposes 
since we are studying the uniqueness of direct sum decompositions up to permutation and isomorphism of summands.

\begin{lem}
\label{lem:Ae-indecom-iff-e-primitive-iff-eAe-local}
Let \(0\neq e\in\Lambda\) be an idempotent. 
Then the following are
equivalent. 
\begin{enumerate}[label=\textup{(\roman*)}] 
	\item\label{item:e-is-primitive} 
		The idempotent \(e\) is primitive.
	\item\label{item:corner-ring-is-local} 
		The only idempotents of 
			\(
				e\Lambda e
			\) 
			are \(0\) and 
			\(
			1_{e\Lambda e} = e
			\).
	\item\label{item:Lambda-e-is-indecomposable} 
		The right \(\Lambda\)-module \(e \Lambda\) is indecomposable.
\end{enumerate}
\end{lem}

\begin{proof}
\ref{item:e-is-primitive} \(\Rightarrow\) \ref{item:corner-ring-is-local}\;\; 
Let \(efe\in e\Lambda e\) be an idempotent. 
We have \(e=efe+(e-efe)\), which implies \(efe=0\) or \(e-efe=0\), because
\(e\) is primitive by assumption, and we are done.

\ref{item:corner-ring-is-local} \(\Rightarrow\)  \ref{item:Lambda-e-is-indecomposable}\;\; 
Let \(e\Lambda=X_{1}\oplus X_{2}\) be
a decomposition of \(e\Lambda\), and let \(f\colon e\Lambda \onto X_{1}\into e\Lambda\),
\(g\colon e\Lambda\onto X_{2}\into e\Lambda\) be the compositions of the canonical projections and inclusions. 
Then \(f,g\) are idempotents in 
\(
	\End_{\rMod{\Lambda}}(e\Lambda) 
	\) 
and, moreover,
\(\id{e\Lambda}=f+g\). 

Using the isomorphism $\Phi$ from \eqref{eqn:isomorphism-Phi} and that $e\neq 0$, we see that $\End_{\rMod{\Lambda}}(e\Lambda) \neq 0$.
Since $\Phi(f),\Phi(g)\in e\Lambda e$ are idempotents, we have 
$\Phi(f),\Phi(g)\in\{ 0,e \}$ by assumption. 
If $\Phi(f)=\Phi(g)=0$, then $f=g=0$ as $\Phi$ is an isomorphism, whence $X_{1} = X_{2} = 0$. But this forces $e\Lambda = 0$, which is a contradiction. 
Thus, without loss of generality, $\Phi(f) = e = \Phi(\id{e\Lambda})$. 
This implies $f = \id{e\Lambda}$ and hence $g = 0$. In particular, this means $X_{2} = 0$ and \(X=\Lambda e\) is indecomposable.

\ref{item:Lambda-e-is-indecomposable} \(\Rightarrow\) \ref{item:e-is-primitive}\;\; 
If \(e=f+g\) where \(f,g\) are orthogonal idempotents,
then \(e\Lambda=(f+g)\Lambda = f\Lambda \oplus g\Lambda\). But then,
without loss of generality, \(f\Lambda=0\) as \(e\Lambda\) is indecomposable by assumption 
and so \(f=0\). Thus, \(e\) is primitive. 
\end{proof}

\subsection{Idempotents in categories}

By an \emph{idempotent} in a category $\CA$ we mean any endomorphism $e\colon X \to X$ for some object $X$ satisfying $e^{2} = e$. We begin with an easy lemma.

\begin{lem}
\label{lem:strong-indecomposable-gives-prim-idempotent}
Let $\CA$ be an additive category with an object $X$. Suppose $X = A\oplus B$ where $\End_{\CA}(A)$ is local. 
Consider the canonical projection \(p\colon X\onto A\) and canonical
inclusion \(i\colon A\into X\). Then the idempotent $e \deff ip \in\End_{\CA}(X)$ is primitive. 
\end{lem}

\begin{proof}
Suppose $e = f+g$ for some orthogonal idempotents $f,g\in\End_{\CA}(X)$. 
Note that $efe = f$. 
A straightforward verification shows $pfi, pgi\in\End_{\CA}(A)$ are orthogonal idempotents and that $\id{A} = pfi + pgi$. 
Since $\End_{\CA}(A)$ is local, by \cref{lem:local-ring-has-0-1-idempotents-only} and without loss of generality, we have that $pfi = 0$. This implies 
$f = efe = i (pfi) p = 0$, and so $e$ is primitive. 
\end{proof}

\begin{defn}
\label{def:has-split-idempotents}
A category \(\CA\) \emph{has split idempotents} 
if, for
each \(X\in\CA\) and every idempotent \(e\in\End_{\CA}(X)\), 
there exists an object \(Y\in\CA\) and morphisms \(r\colon X\to Y\), \(s\colon Y\to X\) 
such that \(e=sr\) and \(rs=\id{Y}\).
\end{defn}

\begin{rem}
\label{rem:retract-section-in-split-idempotent}
With the notation as in \cref{def:has-split-idempotents},
we see that \(r\) is a retraction (hence an epimorphism) and \(s\) is a section (hence a monomorphism). 
\end{rem}

It turns out that any additive category embeds into one that has split idempotents; see Karoubi \cite{Karoubi-algebres-de-Clifford-et-K-theorie}, or B\"{u}hler \cite[\S 6]{Buhler-exact-categories} for a nice exposition.
The following set of equivalent conditions for a category to have split idempotents is well-known for additive categories (see e.g.\ \cite[Rmk.\ 6.2]{Buhler-exact-categories}). 
The same argument, however, also works for preadditive categories.

\begin{prop}
\label{prop:split-idems-iff-idems-admit-kernels-iff-idems-admit-cokernels}
Let \(\CA\) be a preadditive category. 
Then the following are equivalent. 
\begin{enumerate}[label=\textup{(\roman*)}]
	\item\label{item:split-idems} \(\CA\) has split idempotents.
	\item\label{item:idems-kernels} Each idempotent in \(\CA\) admits a kernel in \(\CA\).
	\item\label{item:idems-cokernels} Each idempotent in \(\CA\) admits a cokernel in \(\CA\).
\end{enumerate}
\end{prop}

\begin{proof}
\ref{item:split-idems} \(\Rightarrow\) \ref{item:idems-kernels}\;\; 
Let \(e\colon X\to X\) be an arbitrary idempotent in $\CA$. 
By assumption, 
\(\id{X}-e\) splits and so there exist \(t\colon X\to Z\),
\(u\colon Z\to X\) with \(\id{X}-e=ut\) and \(tu=\id{Z}\).
Recall that $t$ is epic and $u$ is monic (see \cref{rem:retract-section-in-split-idempotent}). 
We claim that \(Z\overset{u}{\to}X\) is a kernel for \(e\). Since \((eu)t=e(ut)=e(\id{X}-e)=0\)
and \(t\) is epic, we see that \(eu=0\). 
Suppose \(a\colon A\to X\)
is a morphism in \(\CA\) such that \(ea=0\). 
Then 
\[
a=a-0=a-ea=(\id{X}-e)a=(ut)a=u(ta),
\]
so \(a\) factors through \(u\). 
Thus, 
$u$ 
is a monomorphism that is a weak kernel of 
\(e\) and hence a kernel of 
\(e\) by \cref{lem:monic-weak-kernel-is-kernel-and-dual}.

\ref{item:idems-kernels} \(\Rightarrow\) \ref{item:split-idems}\;\; 
Let \(e\in\End_{\CA}(X)\) be an idempotent. The idempotent \(\id{X}-e\)
admits a kernel \(Y\deff\Ker(\id{X}-e)\overset{s}{\into}X\). 
As \((\id{X}-e)e=0\), we have that \(e\) factors through \(s\).
Thus, there exists a unique morphism \( r\colon X\to Y\) such that
\(e=s r\), so it suffices to show \( rs=\id{Y}\). As \(s\)
is the kernel of \(\id{X}-e\) we obtain \(s-es=(\id{X}-e)s=0\),
so \(es=s\). Thus, \(s( rs)=(s r)s=es=s\)
and hence \( rs=\id{Y}\) as \(s\) is a monomorphism. 
This shows \(e\) splits. 

The equivalence \ref{item:split-idems} and \ref{item:idems-cokernels} is dual.
\end{proof}

Given an idempotent \(e\colon X\to X\) in an additive category with split idempotents, we can identify two direct summands of \(X\). 
This result is also classical (see e.g.\ Auslander \cite[p.\ 188]{Auslander-Rep-theory-of-Artin-algebras-I}).

\begin{prop}
\label{prop:decomp-of-object-when-idem-splits}
If an additive category \(\CA\) has split idempotents, then for each idempotent \(e\colon X\to X\)
we have \(X = \Ker(e)\oplus\Ker(\id{X}-e)\).
\end{prop}

\begin{proof}
Let \(e\in\End_{\CA}(X)\) be an idempotent. 
Arguing as in the proof of \cref{prop:split-idems-iff-idems-admit-kernels-iff-idems-admit-cokernels}, 
we obtain a commutative diagram 
\[
\begin{tikzcd}[row sep=1cm,column sep=1cm]
&X\arrow[dotted]{dl}[swap]{\exists !p_{1}}\arrow{d}{1_X -e}& \\
X_1\deff\Ker(e)\arrow[hook]{r}{i_1}&X\arrow[dotted]{dl}[swap]{\exists !p_2}\arrow{d}{e}& \\
X_2\deff\Ker(1_X -e)\arrow[hook]{r}[swap]{i_2}&X\arrow{r}[swap]{1_X -e}&X,
\end{tikzcd}
\]
where \(p_{j}i_{j}=\id{X_{j}}\) for \(j=1,2\). We also have \(\id{X}=(\id{X}-e)+e=i_{1}p_{1}+i_{2}p_{2}\).
Furthermore, \(i_{1}p_{1}i_{2}=(\id{X}-e)i_{2}=0\) as
$i_{2} = \ker(\id{X} - e)$, 
so \(p_{1}i_{2}=0\) since \(i_{1}\) is
monic. Similarly, \(p_{2}i_{1}=0\). 
One can then easily check that $X$ satisfies the universal property for the coproduct
\(X_{1}\bincoprod X_{2}\). 
Hence, 
\(X = X_{1}\bincoprod X_{2}=X_{1}\oplus X_{2}=\Ker(e)\oplus\Ker(\id{X}-e)\). 
\end{proof}


\section{Krull-Schmidt categories and semi-perfect rings}
\label{sec:Krull-Schmidt-categories-semi-perfect-rings}

In this section, we follow \cite{Krause-KS-cats-and-projective-covers} in introducing Krull-Remak-Schmidt decompositions and Krull-Schmidt categories (see \cref{def:Krull-Schmidt-category}). 
We include proofs where we think the reader may benefit from some extra detail; otherwise proofs are omitted and can be found in \cite{Krause-KS-cats-and-projective-covers}. See also \cite[App.\ A]{ChenYeZhang-Algebras-of-derived-dimension-zero}.

For this section, let \(\CA\) denote an additive category. 
For an object \(X\in\CA\), we let \(\add{X}\) denote the full subcategory of \(\CA\)
consisting of all direct summands of finite direct sums of copies of \(X\). 
For a ring \(\Lambda\), 
we denote by \(\rproj{\Lambda}\) 
the full subcategory of \(\rMod{\Lambda}\) consisting of finitely
generated projective right \(\Lambda\)-modules.

\begin{rem}
\label{rem:projR-has-split-idems}
The category \(\rproj{\Lambda}\) has split idempotents (see \cite[Exam.\ 2.2(2)]{Krause-KS-cats-and-projective-covers}). 
Indeed, suppose \(e\in\End_{\rproj{\Lambda}}(P)=\End_{\rMod{\Lambda}}(P)\) is an idempotent.
As \(\rMod{\Lambda}\) is abelian, 
the idempotents \(e\) and \(\id{P}-e\) have kernels in \(\rMod{\Lambda}\). 
By \cref{prop:decomp-of-object-when-idem-splits}, 
we know \(P = \Ker(e)\oplus\Ker(\id{P}-e)\) in \(\rMod{\Lambda}\). 
It follows that \(\Ker(e)\) is projective and finitely
generated, so \(e\) has a kernel in \(\rproj{\Lambda}\). 
Therefore, \(\rproj{\Lambda}\) has split idempotents by \cref{prop:split-idems-iff-idems-admit-kernels-iff-idems-admit-cokernels}. 
\end{rem}

\begin{prop}
\label{prop:Kr15-Prop-2-3-hom-functor-induces-equivalence-to-projectives}
Suppose \(X\in\CA\) and set \(\Lambda_{X}\deff\End_{\CA}(X)\).
The additive functor \(H_{X}(-)\deff \Hom_{\CA}(X,-)\colon\CA\to\rMod{\Lambda_{X}}\) induces
a fully faithful additive functor \(\add{X}\to\rproj{\Lambda}_{X}\). 
If \(\CA\) has split idempotents then this is an equivalence. 
For each $X_{0}\in\add{X}$, the induced map 
$F_{X_{0}}\colon \End_{\CA}(X_{0}) \to \End_{\rMod{\Lambda_{X}}}(H_{X}(X_{0}))$ 
is an isomorphism of rings.
\end{prop}
\begin{proof}
The first two assertions follow from \cite[Prop.\ 2.3]{Krause-KS-cats-and-projective-covers}. 
For the last claim, we have that the induced map $F_{X_{0}}$ is a bijective homomorphism of abelian groups since $\Hom_{\CA}(X,-)$ is additive and fully faithful on $\add{X}$. 
Since $\Hom_{\CA}(X,-)$ is a covariant functor, the map $F_{X_{0}}$ 
preserves composition (i.e.\ is multiplicative) and identities (i.e.\ is a map of unital rings). Hence, $F_{X_{0}}$ is a unital ring isomorphism.
\end{proof}

The next definition is a generalisation of an indecomposable module.

\begin{defn}
An object \(X\in\CA\) is called \emph{indecomposable}
if \(X\) is non-zero, and 
if \(X_{1}=0\) or \(X_{2}=0\) whenever there is an isomorphism \(X\iso X_{1}\oplus X_{2}\).
\end{defn}

\begin{rem}
\label{rem:addX-closed-under-direc-summands}
Let $X\in\CA$. Since $\add{X}$ is closed under direct summands, an object $Y\in\add{X}$ is indecomposable in $\add{X}$ if and only if it is indecomposable in $\CA$. 
\end{rem}

The following result is very well-known (see e.g.\ Harada \cite{Harada-apps-of-factor-cats}).

\begin{lem}
\label{lem:local-endo-ring-implies-indecomposable}
If \(\End_{\CA}(X)\) is local, then \(X\) is indecomposable.
\end{lem}

\begin{proof}
Suppose \(\End_{\CA}(X)\) is local and that \(X=X_{1}\oplus X_{2}\). 
Note that \(\End_{\CA}(X)\neq0\) (see \cref{def:local-ring}) so, 
in particular, we see that \(\id{X}\neq0\) and
hence \(X\neq0\). Let \(p_{j}\colon X\onto X_{j}\), respectively,
\(i_{j}\colon X_{j}\into X\), be the canonical projection, respectively,
inclusion for $j=1,2$. Then \(e_{j}\deff i_{j}p_{j}\in\End_{\CA}(X)\) is idempotent
and so we must have \(e_{j}\) is \(0\) or \(\id{X}\) for \(j=1,2\) by Lemma
\ref{lem:local-ring-has-0-1-idempotents-only}.
If \(e_{1}=0=e_{2}\), then we would have \(\id{X}=e_{1}+e_{2}=0\), which
is a contradiction. Hence, without loss of generality, \(e_{1}\neq0\)
and so \(e_{1}=\id{X}\). 
Moreover, this yields $e_{2} = 0$ and hence \(X_{2}=0\), so \(X\) is indecomposable.
\end{proof}

Let us now state the main definition of this section.

\begin{defn}
\label{def:Krull-Schmidt-category}
A finite direct sum decomposition 
\(X = X_{1}\oplus\cdots\oplus X_{n}\)
of $X\in\CA$, 
where \(\End_{\CA}(X_{j})\) is a local ring for all \(1\leq j\leq n\), 
is called a \emph{Krull-Remak-Schmidt decomposition} of \(X\). 
Two Krull-Remak-Schmidt decompositions 
$X= X_{1}\oplus\cdots\oplus X_{n}$ 
and 
$X = Y_{1}\oplus\cdots\oplus Y_{m}$ 
of $X$ are said to be \emph{equivalent} if 
\(m=n\) and there is a permutation \(\sigma\in\sym(n)\) such that 
\(X_{j}\iso Y_{\sigma(j)}\) for all \(1\leq j \leq n\).

If every object in $\CA$ admits a Krull-Remak-Schmidt decomposition, then $\CA$ is known as a \emph{Krull-Schmidt} category. 
\end{defn}

\begin{rem}
\label{rem:zero-object-has-KRS-decomp}
Notice that in an additive category $\CA$, a zero object is the direct sum of an empty (and hence finite) family of objects each having a local endomorphism ring. 
\end{rem}

An immediate consequence of \cref{def:Krull-Schmidt-category} is the following.

\begin{lem}
\label{lem:in-KS-category-indecomposable-iff-local-endo-ring} 
In a Krull-Schmidt category \(\CA\), an object \(X\in\CA\) is indecomposable 
if and only if \(\End_{\CA}(X)\) is local.
\end{lem}

\begin{proof}
\cref{lem:local-endo-ring-implies-indecomposable} treats one direction,
so we suppose \(X\) is indecomposable and show that \(\End_{\CA}(X)\)
is local. As \(\CA\) is Krull-Schmidt, we have a decomposition \(X=X_{1}\oplus\cdots\oplus X_{n}\),
where each \(\End_{\CA}(X_{j})\) is local. 
However, we must have \(n=1\) as $X$ is indecomposable, and so 
\(\End_{\CA}(X)=\End_{\CA}(X_{1})\) is local.
\end{proof}

The next proposition follows from \cite[Prop.\ 4.1]{Krause-KS-cats-and-projective-covers}.

\begin{prop}
\label{prop:Kr15-prop-4-1-semi-perfect-equivalent-conditions}
The following are equivalent for a ring \(\Lambda\). 
\begin{enumerate}[label=\textup{(\roman*)}]
	\item\label{item:projLambda-is-KS} The category \(\rproj{\Lambda}\) is Krull-Schmidt.
	\item\label{item:KRS-decomp-of-Lambda} The right \(\Lambda\)-module \(\Lambda_{\Lambda}\) admits a decomposition 
	\(\Lambda_{\Lambda} = P_1\oplus \cdots \oplus P_n\), where each \(P_{j}\) has a local endomorphism ring.
\end{enumerate}
\end{prop}

\begin{defn}
\label{def:semi-perfect-ring}
If a ring \(\Lambda\) satisfies the equivalent conditions of \cref{prop:Kr15-prop-4-1-semi-perfect-equivalent-conditions},
then it is called \emph{semi-perfect}.
\end{defn}

Notice that for a semi-perfect ring $\Lambda$, the decomposition in \cref{prop:Kr15-prop-4-1-semi-perfect-equivalent-conditions}\ref{item:KRS-decomp-of-Lambda} is a Krull-Remak-Schmidt decomposition in $\rproj{\Lambda}\sse \rMod{\Lambda}$.

As remarked in \cite{Krause-KS-cats-and-projective-covers}, the following theorem is a consequence of the existence and uniqueness of projective covers over a semi-perfect ring.

\begin{thm}
\label{thm:Kr15-Thm-4-2-uniqueness-of-KRS-decomposition}
\emph{\cite[Thm.\ 4.2]{Krause-KS-cats-and-projective-covers}}
Let \(X\in\CA\) be an object. 
Suppose 
\(
X_{1}\oplus\cdots\oplus X_{n}
	= X 
	= Y_{1}\oplus\cdots\oplus Y_{m}
\)
for some objects \(X_{j},Y_{l}\) each having a local endomorphism ring. 
Then 
these two Krull-Remak-Schmidt decompositions are equivalent.
\end{thm}

There are two immediate consequences. 

\begin{cor}
\label{cor:Kr15-cor-4-3-reorder-decomposition}
\emph{\cite[Cor.\ 4.3]{Krause-KS-cats-and-projective-covers}}
If \(\CA\) is Krull-Schmidt and there are two decompositions 
\(
X_{1}\oplus\cdots\oplus X_{n}
	= X
	= X'\oplus X''
\), 
where each \(X_{j}\) is indecomposable, then there exists \(t\leq n\)
such that \(X = X_{1}\oplus\cdots\oplus X_{t}\oplus X'\) (possibly
after reindexing).
\end{cor}

\begin{cor}
\label{cor:Kr15-Cor-4-4-KS-category-iff-split-idems-and-semi-perfect-endo-rings}
\emph{\cite[Thm.\ A.1]{ChenYeZhang-Algebras-of-derived-dimension-zero}, \cite[Cor.\ 4.4]{Krause-KS-cats-and-projective-covers}}
An additive category \(\CA\) 
is Krull-Schmidt, 
if and only if 
it has split idempotents and \(\End_{\CA}(X)\) is semi-perfect for all \(X\in\CA\).
\end{cor}

\begin{proof}
If \(\CA\) has split idempotents and \(\Lambda_{X}\deff\End_{\CA}(X)\) is semi-perfect
for all \(X\in\CA\), then by Propositions \ref{prop:Kr15-Prop-2-3-hom-functor-induces-equivalence-to-projectives}
and \ref{prop:Kr15-prop-4-1-semi-perfect-equivalent-conditions} we have that \(\CA\) is Krull-Schmidt. Indeed, a Krull-Remak-Schmidt decomposition of 
$\Lambda_{X}$ 
as in 
\cref{prop:Kr15-prop-4-1-semi-perfect-equivalent-conditions}\ref{item:KRS-decomp-of-Lambda} yields a Krull-Remak-Schmidt decomposition of $X$ in $\CA$ 
using the 
equivalence $\add{X} \simeq \rproj{\Lambda_{X}}$ 
induced by $\Hom_{\CA}(X,-)$. 

Conversely, suppose \(\CA\) is a Krull-Schmidt category and let \(X\in\CA\)
be arbitrary. By hypothesis, we may take a finite direct sum decomposition \(X=X_{1}\oplus\cdots\oplus X_{n}\)
such that \(\End_{\CA}(X_{j})\) is local for each \(j\). 
We have 
\(
(\Lambda_{X})_{\Lambda_{X}}
	= \End_{\CA}(X)
	\iso \Hom_{\CA}(X,X_{1})\oplus\cdots\oplus\Hom_{\CA}(X,X_{n})
\)
and \(\End_{\rMod{\Lambda_{X}}}(\Hom_{\CA}(X,X_{j}))\iso\End_{\CA}(X_{j})\)
is local (using the fully faithfulness in \cref{prop:Kr15-Prop-2-3-hom-functor-induces-equivalence-to-projectives}). Therefore,
\(\Lambda_{X}\) is a semi-perfect ring, and equivalently \(\rproj{\Lambda}_{X}\)
is Krull-Schmidt by \cref{prop:Kr15-prop-4-1-semi-perfect-equivalent-conditions}. Therefore,
given any finitely generated projective \(\Lambda_{X}\)-module, it
will isomorphic to a direct summand of \(\Lambda_{X}^{m}=\Hom_{\CA}(X,X_{1})^{m}\oplus\cdots\oplus\Hom_{\CA}(X,X_{n})^{m}\)
by \cref{cor:Kr15-cor-4-3-reorder-decomposition}. 
This means that $\Hom_{\CA}(X,-)\colon \add{X} \to \rproj{\Lambda_{X}}$ is also dense, and hence an equivalence. 
The category \(\rproj{\Lambda}_{X}\) has split idempotents 
(see \cref{rem:projR-has-split-idems})
and hence so does \(\add{X}\). 
Moreover, this implies \(\CA\) has split idempotents. 
\end{proof}

We close this section with two results related to semi-perfect rings 
that will be needed for the main result of the paper. 
Although a stronger version of the next lemma can be found in \cite[Thm.\ 27.6]{AndersonFuller-rings-and-cats-of-modules}, 
the proof below is inspired by that of Auslander--Reiten--Smal\o{} \cite[Prop.\ I.4.8]{AuslanderReitenSmalo-rep-theory-of-artin-algebras}.

\begin{lem}
\label{lem:semi-perfect-ring-admits-complete-set-orthog-primitive-idems}
Suppose \(\Lambda\) is semi-perfect and that 
\(\Lambda=P_{1}\oplus\cdots\oplus P_{n}\) 
as right \(\Lambda\)-modules, 
where \(\End_{\rMod{\Lambda}}(P_{j})\) is local for \(1\leq j\leq n\). 
Then \(\Lambda\) admits a complete set \(\{e_{j}\}_{j=1}^{n}\) of primitive orthogonal idempotents, 
such that $P_{j} = e_{j}\Lambda$ and \(e_{j}\Lambda e_{j}\) is local for \(1\leq j\leq n\).
\end{lem}

\begin{proof}
Since \(\Lambda=P_{1}\oplus\cdots\oplus P_{n}\), we may express the
identity \(1_{\Lambda}\) of \(\Lambda\) as \(1_{\Lambda}=e_{1}+\cdots+e_{n}\)
for some elements \(e_{j}\in P_{j}\). 
Fix $l\in\{1,\ldots,n\}$. 
Then we have 
\(
e_{1}e_{l}+\cdots+e_{n}e_{l}
	= (e_{1}+\cdots+e_{n})e_{l} 
	= 1_{\Lambda} \cdot e_{l} 
	= e_{l} 
	\in P_{l}
\).
But since each \(P_{j}\) is a right \(\Lambda\)-module, 
we have \(e_{j}e_{l}\in P_{j}\).
Therefore, using the direct sum decomposition of $\Lambda$, we see that \(e_{j}e_{l}=0\)
for all \(j\neq l\) and that \(e_{l}=e_{l}e_{l}\), 
i.e.\  the set \(\{e_{j}\}_{j=1}^{n}\)
forms a complete set of orthogonal idempotents.

We claim that \(e_{l}\Lambda = P_{l}\). 
Since \(e_{l}\in P_{l}\)
and \(P_{l}\) is a right \(\Lambda\)-module, we immediately see that
\(e_{l}\Lambda\subseteq P_{l}\). 
Conversely, let \(x\in P_{l}\sse \Lambda\) be arbitrary. 
By the same argument as above, \(e_{j}x=0\) for \(j\neq l\) 
and \(x=e_{l}x\in e_{l}\Lambda\), so that
we have the equality \( e_{l}\Lambda=P_{l}\). 
Lastly, using the isomorphism \eqref{eqn:isomorphism-Phi}, we have 
\(
e_{l}\Lambda e_{l} 
	\iso \End_{\rMod{\Lambda}}(e_{l}\Lambda) 
	= \End_{\rMod{\Lambda}}(P_{l}) 
\) 
is local, 
so \(e_{l}=1_{e_{l}\Lambda e_{l}}\neq0\) and \(e_{l}\) is primitive by combining Lemmas~\ref{lem:local-ring-has-0-1-idempotents-only} and \ref{lem:Ae-indecom-iff-e-primitive-iff-eAe-local}. 
\end{proof}

The following result is a consequence of Jacobson \cite[Thm.\ III.10.2]{Jacobson-structure-of-rings}. 
Jacobson states the result with the assumption that the primitive 
idempotents give rise to local corner rings. 
Here, we suppose that the ring $\Lambda$ itself is semi-perfect
so that \(\rproj{\Lambda}\) is Krull-Schmidt, and this implies the hypothesis
needed in \cite{Jacobson-structure-of-rings}. 
Furthermore, we note that the proof below appears in a pre-published version of Liu--Ng--Paquette 
\cite{LiuNgPaquette-almost-split-sequences-and-approximations}.

\begin{prop} 
\label{prop:semi-perfect-ring-the-complete-set-orthog-primitive-idems-is-essentially-unique}
Let \(\Lambda\) be a 
semi-perfect ring, and suppose 
\(\{e_{j}\}_{j=1}^{n},
\{f_{l}\}_{l=1}^{m}
\)
are complete sets of primitive orthogonal idempotents in $\Lambda$. 
Then \(m=n\),
and there exists a permutation \(\sigma\in\sym(n)\) and an invertible
element \(a\in\Lambda\) such that \(f_{\sigma(j)}=ae_{j}a^{-1}\) for \(1\leq j\leq n\).
\end{prop}

\begin{proof}
Since \(\{e_{j}\}_{j=1}^{n},
\{f_{l}\}_{l=1}^{m}
\)
are complete sets of primitive orthogonal idempotents in $\Lambda$, 
we obtain \(e_{1} \Lambda \oplus\cdots\oplus e_{n}\Lambda = \Lambda=f_{1}\Lambda \oplus\cdots\oplus f_{m}\Lambda\),
where \(e_{j}\Lambda \) and \(f_{l}\Lambda\) are indecomposable
by \cref{lem:Ae-indecom-iff-e-primitive-iff-eAe-local}. 
Moreover, since \(\Lambda\) is
semi-perfect we know \(\rproj{\Lambda}\) is Krull-Schmidt and so by 
\cref{prop:Kr15-prop-4-1-semi-perfect-equivalent-conditions} we have that 
\(\End_{\rMod{\Lambda}}(e_{j}\Lambda)\) and
\(\End_{\rMod{\Lambda}}(f_{l}\Lambda)\) are local rings. 
Therefore, we may apply \cref{thm:Kr15-Thm-4-2-uniqueness-of-KRS-decomposition} so that \(m=n\)
and \(e_{j}\Lambda=f_{\sigma(j)}\Lambda \), for \(j=1,\ldots,n\), for
some permutation \(\sigma\in\sym(n)\). 

Hence, there exist \(b_{j},c_{j}\in\Lambda\), 
such that 
\(
e_{j}
	= f_{\sigma(j)} b_{j}
	= f_{\sigma(j)} b_{j} e_{j}
\) 
and 
\(
f_{\sigma(j)}
	= e_{j} c_{j} f_{\sigma(j)}
\)
for each \(1\leq j\leq n\). 
In particular, \(e_{j}=(e_{j}c_{j}f_{\sigma(j)})(f_{\sigma(j)}b_{j}e_{j})\)
and 
\begin{equation}\label{eqn:fsigmaj}
f_{\sigma(j)}=(f_{\sigma(i)}b_{j}e_{j})(e_{j}c_{j}f_{\sigma(j)}). 
\end{equation}
Set
\(a\deff\sum_{j=1}^{n}f_{\sigma(j)}b_{j}e_{j}\) 
and 
\(a^{-1}\deff\sum_{l=1}^{n}e_{l}c_{l}f_{\sigma(l)}\).
We observe that 
\begin{align*}
a\cdot a^{-1}
	&= \left(\sum_{j=1}^{n}f_{\sigma(j)}b_{j}e_{j}\right)
	\left(\sum_{l=1}^{n}e_{l}c_{l}f_{\sigma(l)}\right) \\
	&= \sum_{j=1}^{n}(f_{\sigma(j)}b_{j}e_{j})(e_{j}c_{j}f_{\sigma(j)}) && \text{as \(\{e_{j}\}_{j=1}^{n}\) is orthogonal}\\
	&= \sum_{j=1}^{n}f_{\sigma(j)} &&\text{using \eqref{eqn:fsigmaj}}\\
	&= 1_{\Lambda}  &&\text{as $\{f_{j}\}_{j=1}^{m}$ is complete}. 
\end{align*}
Similarly, one can show \(a^{-1}a=1_{\Lambda}\). 
Finally,
\[
ae_{r}a^{-1}
	= \left(\sum_{j=1}^{n}f_{\sigma(j)}b_{j}e_{j}\right)e_{r}\left(\sum_{l=1}^{n}e_{l}c_{l}f_{\sigma(l)}\right)
	= (f_{\sigma(r)}b_{r}e_{r}) e_{r} (e_{r}c_{r}f_{\sigma(r)})
	= f_{\sigma(r)},
\]
again using that \(\{e_{j}\}_{j=1}^{n}\) is orthogonal and \eqref{eqn:fsigmaj}, 
and this finishes the proof.
\end{proof}


\section{Subobjects, the bi-chain condition and \texorpdfstring{$\Hom$}{Hom}-finiteness}
\label{sec:subobjects-bichain-Hom-finite}

The goal of this section is to see that the endomorphism ring of any object in a $\Hom$-finite additive category (see \cref{def:Hom-finite}) is semi-perfect (see \cref{cor:Hom-finite-implies-endo-rings-are-all-semi-perfect}). This is shown via Atiyah's bi-chain conditions (see \cref{def:bi-chain}).

\begin{defn}
\cite[p.\ 539]{Krause-KS-cats-and-projective-covers} 
Suppose $\CB$ is an abelian category and let $X\in\CA$. 
Two monomorphisms \(a\colon X_{1} \into X\) and \(b\colon X_{2}\into X\) are \emph{equivalent} if there is an isomorphism \(c\colon X_{1}\overset{\iso}{\longrightarrow}X_{2}\)
such that \(bc=a\). This is an equivalence relation on the collection of monomorphisms in $\CB$ with codomain $X$, and
an equivalence class of this relation is called a \emph{subobject} of \(X\). By abuse of notation we just write \(a\colon X_{1} \into X\) for the equivalence class containing the monomorphism $a$. 

Given two subobjects \(A\overset{f}{\into}X\) and \(B\overset{g}{\into}X\)
of \(X\), we say that \(A\) is \emph{contained in} \(B\) (denoted \(A\subseteq B\))
if there is a (necessarily monic) morphism \(h\colon A\to B\) such that \(f=gh\).
In this way, we obtain a partial order on the collection of subobjects
of \(X\).
\end{defn}

We assume the following throughout this section.

\begin{setup}\label{setup:bi-chain}
We denote by $\CB$ an abelian category. 
We make the implicit assumption that the collection of
subobjects of an object \(X\in\CB\) is a set.
\end{setup}

\begin{rem}
The set-theoretic restriction in Setup~\ref{setup:bi-chain} is not so strong. 
For example, the condition is satisfied if $\CB$ has a generator 
(see Freyd \cite[Prop.\ 3.35]{Freyd-abelian-cats}), 
or if $\CB$ is skeletally small. 
In particular, the category of all (right) modules over a ring falls into Setup~\ref{setup:bi-chain}.
\end{rem}

We call \(X\in\CB\) \emph{simple} if $X\neq 0$ and 
its only subobjects are \(0\) and \(X\) itself (see \cite[p.\ 539]{Krause-KS-cats-and-projective-covers}). 
Given a subobject \(A\overset{f}{\into}X\), 
we denote by \(X/A\) the codomain of
the cokernel map \(\cok f\colon X\onto\Cok f\).

\begin{defn}
\label{def:finite-length-object-composition-series}
\cite[p.\ 547]{Krause-KS-cats-and-projective-covers}
An object \(X\in\CB\) is said to have 
\emph{finite length} if there is a finite chain (called a \emph{composition series})
\[
0=X_{0}\subseteq X_{1}\subseteq\cdots\subseteq X_{n-1}\subseteq X_{n}=X
\]
of subobjects of \(X\) such that each successive quotient \(X_{j+1}/X_{j}\)
is simple.
\end{defn}

Now we recall the bi-chain condition in an abelian category as introduced in \cite{Atiyah-KS-theorem-with-apps-to-sheaves}.

\begin{defn}
\cite[p.\ 310]{Atiyah-KS-theorem-with-apps-to-sheaves}, \cite[p.\ 546]{Krause-KS-cats-and-projective-covers} 
\label{def:bi-chain}
A \emph{bi-chain} in $\CB$ is a sequence of morphisms 
\begin{equation}\label{eqn:bi-chain}
(\begin{tikzcd}
X_n \arrow[two heads]{r}{\alpha_n} & X_{n+1}\arrow[hook]{r}{\beta_n}&X_n\end{tikzcd})_{n\geq 0}
\end{equation}
in \(\CB\), 
for which \(\alpha_{n}\) is epic and 
\(\beta_{n}\) is monic for all \(n\geq0\). 
An object \(X\in\CB\) is said to \emph{satisfy the bi-chain condition} 
if for any bi-chain \eqref{eqn:bi-chain} 
with \(X_{0}=X\), there exists \(N\geq0\) such that \(\alpha_{n},\beta_{n}\)
are isomorphisms \(\forall n\geq N\).
\end{defn}

If \(X\in\CB\) satisfies the bi-chain condition, 
then \(X\) is indecomposable if and only if \(\End_{\CB}(X)\) is local; 
see \cite[Prop.\ 5.4]{Krause-KS-cats-and-projective-covers}, also  \cite[Lem.\ 6]{Atiyah-KS-theorem-with-apps-to-sheaves}.

\begin{lem}
\label{lem:Kr15-lem-5-1-finite-length-implies-bi-chain-condition}
\emph{\cite[Lem.\ 5.1]{Krause-KS-cats-and-projective-covers}}
If \(X\in\CB\) has finite length, 
then \(X\) satisfies the bi-chain condition.
\end{lem}

\begin{defn}
\label{def:Hom-finite}
\cite[p.\ 547]{Krause-KS-cats-and-projective-covers} 
We call a category \(\CA\) a 
\emph{\(\Hom\)-finite \(\ring\)-linear category} 
if \(\ring\) is a commutative ring, 
such that 
\(\CA\) is \(\ring\)-linear and for which \(\Hom_{\CA}(X,Y)\)
is a finite length \(\ring\)-module for all \(X,Y\in\CA\). 
If the ring $\ring$ or its existence is understood, then we more simply say that $\CA$ is \emph{$\Hom$-finite}. 
\end{defn}

Any object \(X\) of a $\Hom$-finite abelian category 
satisfies the bi-chain condition; 
see \cite[Lem.\ 5.2]{Krause-KS-cats-and-projective-covers}.

\begin{thm}
\label{thm:Kr15-thm-5-5-object-satisfying-bi-chain-has-KRS-decomposition}
\emph{\cite[Thm.\ 5.5]{Krause-KS-cats-and-projective-covers}}, \emph{\cite[Lem.\ 4]{Atiyah-KS-theorem-with-apps-to-sheaves}} 
Suppose \(X\in\CB\) satisfies the bi-chain condition.
Then \(X\) admits a finite direct sum decomposition 
\(X = X_{1}\oplus\cdots\oplus X_{n}\), 
with \(\End_{\CB}(X_{j})\) local for all \(1\leq j\leq n\).
\end{thm}

Putting together the results above we derive the following.

\begin{cor}
\label{cor:Hom-finite-implies-endo-rings-are-all-semi-perfect}
If \(\CA\) is a \(\Hom\)-finite $\ring$-linear category, 
then \(\End_{\CA}(X)\) is semi-perfect for each \(X\in\CA\).
\end{cor}

\begin{proof}
If \(\CA\) is a \(\Hom\)-finite $\ring$-linear category, then \(\Lambda_{X}\deff\End_{\CA}(X)\) is
a finite length \(\ring\)-module. 
That is,
\(\Lambda_{X}\) is a finite length object in the abelian category \(\rMod{k}\), 
and hence also a finite length object in \(\rMod{\Lambda_{X}}\). 
Therefore, by
\cref{lem:Kr15-lem-5-1-finite-length-implies-bi-chain-condition} we know 
\(\Lambda_{X}\) satisfies the bi-chain
condition in \(\rMod{\Lambda_{X}}\). 
This implies that \(\Lambda_{X}\) admits a decomposition
into a finite direct sum of right $\Lambda_{X}$-modules with local endomorphism rings 
by \cref{thm:Kr15-thm-5-5-object-satisfying-bi-chain-has-KRS-decomposition},
which is precisely condition \ref{item:KRS-decomp-of-Lambda} of \cref{prop:Kr15-prop-4-1-semi-perfect-equivalent-conditions}. 
Hence, \(\Lambda_{X}\) is semi-perfect.
\end{proof}


\section{The main theorem}
\label{sec:main-theorem}

We are now in position to state and prove the theorem we have been building too. It is well-known and stated in several places, e.g.\ 
\cite[\S 2.2]{Ringel-tame-algebras-and-integral-quadratic-forms}, 
but we could not find a proof. 
The equivalence below is 
asserted in \cite[\S I.3.2]{Happel-triangulated-cats-in-rep-theory}  without an explicit $\Hom$-finiteness assumption, 
but we believe this may be in error. 
It was a desire to understand this that motivated this note.

We note that the equivalence of \ref{item:main-thm-KS-category} and \ref{item:main-thm-idempotents-split} in \cref{thm:main-theorem} follows from \cref{cor:Kr15-Cor-4-4-KS-category-iff-split-idems-and-semi-perfect-endo-rings} once we know each endomorphism ring arising from a $\Hom$-finite category is semi-perfect (see \cref{cor:Hom-finite-implies-endo-rings-are-all-semi-perfect}). 
We give a more pedestrian proof of 
\ref{item:main-thm-KS-category} implies \ref{item:main-thm-idempotents-split} below.

\begin{thm}
\label{thm:main-theorem}
Let \(\CA\) be a \(\Hom\)-finite \(\ring\)-linear category. 
Then the following are equivalent. 
\begin{enumerate}[label=\textup{(\roman*)}]
	\item\label{item:main-thm-KS-category} 
	\(\CA\) is a Krull-Schmidt category.
	\item\label{item:main-thm-idempotents-split} 
	$\CA$ has split idempotents.
	\item\label{item:main-thm-local-iff-indecomp} 
	For any object \(Y\in\CA\), the ring \(\End_{\CA}(Y)\) is local if and only if \(Y\) is indecomposable.
\end{enumerate}
Furthermore, in this case, an object $X\in\CA$ admits a Krull-Remak-Schmidt decomposition $X = X_{1}\oplus \cdots \oplus X_{n}$ in $\CA$ 
if and only if 
$\End_{\CA}(X)$ admits a complete set of primitive orthogonal idempotents of size $n$. 
\end{thm}

\begin{proof}
Throughout this proof we use that, since $\CA$ is $\Hom$-finite, the endomorphism ring $\Lambda_{X} \deff \End_{\CA}(X)$ of each object $X\in\CA$ is semi-perfect by \cref{cor:Hom-finite-implies-endo-rings-are-all-semi-perfect}. 
Furthermore, this implies \(\rproj{\Lambda}_{X}\) is Krull-Schmidt by 
\cref{prop:Kr15-prop-4-1-semi-perfect-equivalent-conditions}.

\ref{item:main-thm-KS-category}
	$\Rightarrow$ \ref{item:main-thm-idempotents-split}\;\;
	Fix an object \(X\in\CA\). 
If \(X=0\) then any idempotent \(e\in\End_{\CA}(X)\)
is trivially split, so assume \(X\neq0\). 
Since $\CA$ is Krull-Schmidt, there is a Krull-Remak-Schmidt decomposition
$X = X_{1}\oplus \cdots \oplus X_{n}$ of $X$ in $\CA$. 
Consider the canonical projections \(p_{j}\colon X\onto X_{j}\) 
and inclusions \(i_{j}\colon X_{j}\into X\). 
Putting 
\(e_{j}\deff i_{j}p_{j}\) for each \(1\leq j\leq n\), we see that
\(\{e_{j}\}_{j=1}^{n}\) forms a complete set of orthogonal
idempotents of \(\Lambda_{X}\). 
Each idempotent $e_{j}$ is primitive by \cref{lem:strong-indecomposable-gives-prim-idempotent}, and hence 
\(\Lambda_{X}
	= e_{1}\Lambda_{X} \oplus\cdots\oplus e_{n}\Lambda_{X}
\) 
is a decomposition into indecomposable right $\Lambda_{X}$-modules.

Suppose \(e\colon X\to X\) is an idempotent morphism. 
If $e=0$, then it trivially splits, so we may assume $e\neq0$. 
By \cref{example:idempotents}\ref{item:1-e-is-idempotent}, we have 
\(
e \Lambda_{X} \oplus (\id{X}-e)\Lambda_{X}
	=\Lambda_{X}
	=\bigoplus_{j=1}^{n}e_{j}\Lambda_{X}
\).
Therefore, 
working in the Krull-Schmidt
category \(\rproj{\Lambda}_{X}\), 
we see that 
\(e \Lambda_{X}=\bigoplus_{j=1}^{t} e_{j} \Lambda_{X}\) 
and 
\((\id{X}-e)\Lambda_{X}=\bigoplus_{j=t+1}^{n} e_{j} \Lambda_{X}\)
for some \(1\leq t\leq n\) (possibly after reindexing) by \cref{cor:Kr15-cor-4-3-reorder-decomposition}. 
Therefore, 
we may express 
\(e = e_{1}r_{1}+\cdots+e_{t}r_{t}\) and 
\(1_{\Lambda_{X}} - e = \id{X} - e = e_{t+1}r_{t+1}+\cdots+e_{n}r_{n}\)
for some \(r_{j}\in\Lambda_{X}\), where $1\leq j\leq n$. 
We claim that 
\(\{g_{j}\deff e_{j}r_{j} \}_{j=1}^{n} \sse \Lambda_{X} \)
is a complete set of primitive orthogonal idempotents
and that 
\(g_{j}\Lambda_{X}=e_{j}\Lambda_{X}\). 
First, note that 
\(g_{j} \Lambda_{X} 
	= e_{j}r_{j}\Lambda_{X}
	\subseteq e_{j}\Lambda_{X}
\) for all $1\leq j\leq n$. 
Now fix $j\in\{1,\ldots, t\}$. 
If we have 
\(x  \in e_{j}\Lambda_{X}\subseteq e\Lambda_{X}\),
then it satisfies $ex = x$ and $(\id{X} - e)x = 0$. 
Hence, the identity 
\[
x 
	= 1_{\Lambda_{X}} \cdot x
	= e_{1}r_{1} x + \cdots + e_{n}r_{n} x
	= g_{1} x + \cdots + g_{n} x
\]
implies 
\(x = g_{j} x\) and \(g_{l} x=0\) 
for all \(l\neq j\), 
using \(\Lambda_{X}=\bigoplus_{j=1}^{n}e_{j}\Lambda_{X}\), 
as \(g_{l}x\in e_{l}\Lambda_{X}\) and \(x \in e_{j}\Lambda_{X}\).
In particular, this yields \(e_{j}\Lambda_{X}\subseteq g_{j}\Lambda_{X}\) and so \(g_{j}\Lambda_{X}= e_{j} \Lambda_{X}\) for $1\leq j \leq t$. 
Furthermore, we also see that 
\(g_{l}g_{j}=0\) for \(l\neq j\) and 
\(g_{j}^{2}=g_{j}\) by choosing $x=g_{j}$. 
A similar argument yields the same conclusions for $j\in\{t+1,\ldots,n\}$. 
Moreover, by \cref{lem:Ae-indecom-iff-e-primitive-iff-eAe-local} 
we deduce that 
\(g_{j}\) is primitive as 
\(g_{j}\Lambda_{X}=e_{j}\Lambda_{X}\) is indecomposable. 
Hence, 
\(\{g_{j}\}_{j=1}^{n}\)
is a set of primitive
orthogonal idempotents in \(\Lambda_{X}\), and it is clear that it is complete.

By \cref{prop:semi-perfect-ring-the-complete-set-orthog-primitive-idems-is-essentially-unique}
there exists an invertible element \(a\in\Lambda_{X}\) and a permutation \(\sigma\in\sym(n)\), such
that \(g_{\sigma(j)}=ae_{j}a^{-1}\) for all \(j=1,\ldots,n\). 
But, by inspecting 
the proof of \cref{prop:semi-perfect-ring-the-complete-set-orthog-primitive-idems-is-essentially-unique},
we observe that \(\sigma(j)=j\) for each $j$ 
as \(e_{j}\Lambda_{X}=g_{j}\Lambda_{X}\).
Therefore, \(g_{j}=ae_{j}a^{-1}\) 
for \(1\leq j\leq n\).
Define \(Y\deff X_{1}\oplus\cdots\oplus X_{t}\) 
and \(Z\deff X_{t+1}\oplus\cdots\oplus X_{n}\), 
then \(X=Y\oplus Z\). 
Define morphisms  
$p\deff(\, p_{1}\, \cdots \, p_{t}\,)^{T}a^{-1}\colon X\to Y$ 
and 
$i\deff a (\, i_{1}\, \cdots \, i_{t}\,) \colon Y\to X$. 
Then
\[
ip 
	= a (\, i_{1}\, \cdots \, i_{t}\,)\circ (\, p_{1}\, \cdots \, p_{t}\,)^{T}a^{-1}
	= \sum_{j=1}^{t} ai_{j}p_{j}a^{-1}
	= \sum_{j=1}^{t}ae_{j}a^{-1}
	= \sum_{j=1}^{t}g_{j}
	= e,
\]
and 
\(
pi
\)
is the $(t\times t)$-diagonal matrix with diagonal 
$(p_{1}i_{1}, \ldots, p_{t}i_{t}) = (\id{X_{1}},\ldots, \id{X_{t}})$, 
i.e.\ $pi = \id{Y}$. 
That is, we have shown \(e\) splits and hence $\CA$ has split idempotents.

\ref{item:main-thm-idempotents-split} 
	$\Rightarrow$ \ref{item:main-thm-local-iff-indecomp}\;\;
Let \(Y\in\CA\) be arbitrary and note \(\Lambda_{Y}=\End_{\CA}(Y) = \End_{\add{Y}}(Y)\). 
As \(\CA\) has split idempotents,  
there is an equivalence \(\add{Y}\simeq\rproj{\Lambda}_{Y}\) by \cref{prop:Kr15-Prop-2-3-hom-functor-induces-equivalence-to-projectives}. 
In particular, $\add{Y}$ is a Krull-Schmidt category.  
Using this, 
\cref{rem:addX-closed-under-direc-summands} 
and  
\cref{lem:in-KS-category-indecomposable-iff-local-endo-ring}, 
we know $Y$ is indecomposable in $\CA$, if and only if it is indecomposable in $\add{Y}$, if and only if $\Lambda_{Y}$ is local. 
	
\ref{item:main-thm-local-iff-indecomp}
	$\Rightarrow$ \ref{item:main-thm-KS-category}\;\;
Fix an object \(X\in\CA\). 
If \(X=0\) then it trivially has a Krull-Remak-Schmidt decomposition (see \cref{rem:zero-object-has-KRS-decomp}). 
Thus, assume \(X\neq0\). 
As \(\Lambda_{X} = \End_{\CA}(X)\) is semi-perfect, there is a direct sum decomposition 
\(
(\Lambda_{X})_{\Lambda_{X}} 
	= P_{1}\oplus\cdots\oplus P_{n}\) with \(\End_{\rMod{\Lambda_{X}}}(P_{j})
\)
local for each \(1\leq j\leq n\). 
By \cref{lem:semi-perfect-ring-admits-complete-set-orthog-primitive-idems}, there is a complete set $\{ f_{j} \}_{j=1}^{n}\sse \Lambda_{X}$ of primitive orthogonal idempotents, such that $P_{j} = f_{j}\Lambda_{X}$. 
As in \cref{prop:Kr15-Prop-2-3-hom-functor-induces-equivalence-to-projectives}, 
put $H_{X}(-) = \Hom_{\CA}(X,-)$ and recall that there is a ring isomorphism
$\End_{\CA}(X_{0}) \to \End_{\rMod{\Lambda_{X}}}(H_{X}(X_{0}))$ 
induced by $H_{X}(-)$ 
for each object $X_{0}\in\add{X}$.

We prove by induction on $n$ that $X$ admits a Krull-Remak-Schmidt decomposition 
$X = X_{1} \oplus \cdots \oplus X_{n}$
of length $n$ 
in $\CA$. 
If $n=1$, then $\Lambda_{X} = P_{1}$ has a local endomorphism ring. 
Then the ring isomorphism 
$
\End_{\CA}(X) 
	\iso \End_{\rMod{\Lambda_{X}}}(\Lambda_{X})
$ 
implies $\End_{\CA}(X)$ is local, so we set $X_{1}\deff X$ and we are done in this case.

Now suppose $n\geq 2$ and that the claim holds for positive integers $m<n$. 
If $X$ is indecomposable, then $\Lambda_{X} = \End_{\CA}(X)$ is local by assumption \ref{item:main-thm-local-iff-indecomp}. 
This would imply $1_{\Lambda}$ is primitive by \cref{lem:local-ring-has-0-1-idempotents-only}, 
and then in turn force $n=1$ by combining \cref{lem:semi-perfect-ring-admits-complete-set-orthog-primitive-idems} and \cref{prop:semi-perfect-ring-the-complete-set-orthog-primitive-idems-is-essentially-unique}, leading to a contradiction. 
Hence, $X = Y_{1} \oplus Y_{2}$ for some non-zero objects $Y_{j}\in\add{X}$. 
We have 
$
H_{X}(Y_{1})\oplus H_{X}(Y_{2})
	\iso H_{X}(X) 
	= \Lambda_{X} 
	= \bigoplus_{j=1}^{n}f_{j} \Lambda_{X}
$
in the Krull-Schmidt category \(\rproj{\Lambda}_{X}\).
Thus, 
\begin{equation}
\label{eqn:decomp-HY1}
H_{X}(Y_{1})=\bigoplus_{j=1}^{m_{1}}f_{j} \Lambda_{X}
\end{equation}
and 
\(H_{X}(Y_{2})=\bigoplus_{j=m_{1}+1}^{n}f_{j} \Lambda_{X}\) 
for some \(0\leq m_{1}\leq n\) 
(possibly after reindexing)
by \cref{cor:Kr15-cor-4-3-reorder-decomposition}. 
Since there is the non-zero canonical projection of $X$ onto its summand $Y_{j}$ for $j=1,2$, we must actually have that $1\leq m_{1} \leq n-1$. 
Define $f\deff f_{1} + \cdots + f_{m_{1}}$. 
Then
\begin{align*}
\Lambda_{Y_{1}} 
	&= \End_{\CA}(Y_{1}) \\
	&\iso \End_{\rproj{\Lambda_{X}}}(H_{X}(Y_{1})) &&\text{by \cref{prop:Kr15-Prop-2-3-hom-functor-induces-equivalence-to-projectives} since $Y_{1}\in\add{X}$}\\
	&= \End_{\rMod{\Lambda_{X}}}(H_{X}(Y_{1}))  &&\text{as $\rproj{\Lambda_{X}}\sse\rMod{\Lambda_{X}}$ is a full subcategory}\\
	&= \End_{\rMod{\Lambda_{X}}}\left(\bigoplus_{j=1}^{m_{1}}f_{j} \Lambda_{X}\right) &&\text{using \eqref{eqn:decomp-HY1}}\\
	&\iso f\Lambda_{X}f&&\text{using \eqref{eqn:isomorphism-Phi},}
\end{align*}
where the two isomorphisms are ring isomorphisms. 
It is straightforward to check that $\{ f_{j} \}_{j=1}^{m_{1}}\sse f\Lambda_{X}f$ is a complete set of primitive orthogonal idempotents, 
and hence $\Lambda_{Y_{1}}$ also admits a complete set of primitive orthogonal idempotents of size $m_{1}$. 
Similarly, $\Lambda_{Y_{2}}$ has a complete set of primitive orthogonal idempotents of size $m_{2}\deff n - m_{1}$, where $1\leq m_{2} \leq n-1$. 
Therefore, we can apply our induction hypothesis to $Y_{1}$ and $Y_{2}$, which produces a Krull-Remak-Schmidt decomposition 
$X 
	= Y_{1}\oplus Y_{2} 
	= X_{1}\oplus \cdots \oplus X_{m_{1}} 
	\oplus X'_{1}  \oplus \cdots \oplus X'_{m_{2}}
$
of $X$ in $\CA$ of length $m_{1} + m_{2} = n$. 
\end{proof}


{\setstretch{1}\begin{acknowledgements}
This work began during my Ph.D.\ at University of Leeds. 
I would like to thank Bethany R.\ Marsh, whose guidance and patience at that time helped me understand the content of this note. 
I am very grateful to Raphael Bennett-Tennenhaus who made me aware of the contributions of R.\ E.\ Remak to the theory of direct sum decompositions and who gave me valuable comments. 
I also thank Xiao-Wu Chen, Alberto Facchini and Henning Krause for helpful discussions and email communications in the preparation of this note. 

During the preparation of this note I have been financially supported by the following, for which I am also grateful:
a University of Leeds 110 Anniversary Research Scholarship, 
the Engineering and Physical Sciences Research Council (grant EP/P016014/1), 
an Early Career Fellowship from the London Mathematical Society with support from the Heilbronn Institute for Mathematical Research (grant ECF-1920-57), 
a DNRF Chair from the Danish National Research Foundation (grant DNRF156), 
a Research Project 2 from the Independent Research Fund Denmark (grant 1026-00050B), 
the Aarhus University Research Foundation (grant AUFF-F-2020-7-16). 
\end{acknowledgements}}

\newpage
{\setstretch{1}\bibliographystyle{mybstwithlabels}
\bibliography{references}}

\begin{thebibliography}{{Mac}98}
\expandafter\ifx\csname url\endcsname\relax
  \def\url#1{\texttt{#1}}\fi
\expandafter\ifx\csname doi\endcsname\relax
  \def\doi#1{\burlalt{doi:#1}{http://dx.doi.org/#1}}\fi
\expandafter\ifx\csname urlprefix\endcsname\relax\def\urlprefix{URL }\fi
\expandafter\ifx\csname href\endcsname\relax
  \def\href#1#2{#2}\fi
\expandafter\ifx\csname burlalt\endcsname\relax
  \def\burlalt#1#2{\href{#2}{#1}}\fi

\bibitem[AF92]{AndersonFuller-rings-and-cats-of-modules}
F.~W. Anderson and K.~R. Fuller.
\newblock {\em Rings and categories of modules}, volume~13 of Graduate Texts in
  Mathematics.
\newblock Springer-Verlag, New York, second edition, 1992.

\bibitem[Alu09]{Aluffi-Chapter0}
P.~Aluffi.
\newblock {\em Algebra: chapter 0}, volume 104 of Graduate Studies in
  Mathematics.
\newblock American Mathematical Society, Providence, RI, 2009.

\bibitem[ARS95]{AuslanderReitenSmalo-rep-theory-of-artin-algebras}
M.~Auslander, I.~Reiten, and S.~O. Smal{\o}.
\newblock {\em Representation theory of {A}rtin algebras}, volume~36 of
  Cambridge Studies in Advanced Mathematics.
\newblock Cambridge Univ. Press, Cambridge, 1995.

\bibitem[ASS06]{AssemSimsonSkowronski-Vol1}
I.~Assem, D.~Simson, and A.~Skowro{\'{n}}ski.
\newblock {\em Elements of the representation theory of associative algebras.
  {V}ol. 1. {T}echniques of representation theory}, volume~65 of London
  Mathematical Society Student Texts.
\newblock Cambridge Univ. Press, Cambridge, 2006.

\bibitem[Ati56]{Atiyah-KS-theorem-with-apps-to-sheaves}
M.~Atiyah.
\newblock {\em On the {K}rull-{S}chmidt theorem with application to sheaves}.
\newblock Bull. Soc. Math. France, {\bf 84}:307--317, 1956.

\bibitem[Aus74]{Auslander-Rep-theory-of-Artin-algebras-I}
M.~Auslander.
\newblock {\em Representation theory of {A}rtin algebras. {I}}.
\newblock Comm. Algebra, {\bf 1}:177--268, 1974.

\bibitem[B{\"{u}}h10]{Buhler-exact-categories}
T.~B{\"{u}}hler.
\newblock {\em Exact categories}.
\newblock Expo. Math., {\bf 28}(1):1--69, 2010.

\bibitem[CYZ08]{ChenYeZhang-Algebras-of-derived-dimension-zero}
X.-W. Chen, Y.~Ye, and P.~Zhang.
\newblock {\em Algebras of derived dimension zero}.
\newblock Comm. Algebra, {\bf 36}(1):1--10, 2008.

\bibitem[D{\"{u}}r86]{Dur-Mobius-functions-incidence-algebras-and-power-series-representations}
A.~D{\"{u}}r.
\newblock {\em M{\"{o}}bius functions, incidence algebras and power series
  representations}, volume 1202 of Lecture Notes in Mathematics.
\newblock Springer-Verlag, Berlin, 1986.

\bibitem[Fac03]{Facchini-the-KS-theorem}
A.~Facchini.
\newblock {\em The {K}rull-{S}chmidt theorem}.
\newblock In Handbook of algebra, {V}ol. 3, volume~3 of Handb. Algebr., pages
  357--397. Elsevier/North-Holland, Amsterdam, 2003.

\bibitem[Fre03]{Freyd-abelian-cats}
P.~J. Freyd.
\newblock {\em Abelian categories}.
\newblock Repr. Theory Appl. Categ., {\bf 3}:1--190, 2003.

\bibitem[FS79]{FrobeniusStickelberger-Uber-gruppen-von-vertauschbaren-elementen}
H.~Frobenius and H.~Stickelberger.
\newblock {\em {\"{U}}ber {G}ruppen von vertauschbaren {E}lementen}.
\newblock J. Reine Angew. Math., {\bf 86}:217--262, 1879.

\bibitem[Hap88]{Happel-triangulated-cats-in-rep-theory}
D.~Happel.
\newblock {\em Triangulated categories in the representation theory of
  finite-dimensional algebras}, volume 119 of London Math. Soc. Lecture Note
  Ser.
\newblock Cambridge Univ. Press, Cambridge, 1988.

\bibitem[Har74]{Harada-apps-of-factor-cats}
M.~Harada.
\newblock {\em Applications of factor categories to completely indecomposable
  modules}.
\newblock Publ. D{\'{e}}p. Math. (Lyon), {\bf 11}(2):19--104, 1974.

\bibitem[Jac64]{Jacobson-structure-of-rings}
N.~Jacobson.
\newblock {\em Structure of rings}.
\newblock American Mathematical Society Colloquium Publications, Vol. 37.
  Revised edition. American Mathematical Society, Providence, R.I., 1964.

\bibitem[Kar68]{Karoubi-algebres-de-Clifford-et-K-theorie}
M.~Karoubi.
\newblock {\em Alg\`ebres de {C}lifford et {$K$}-th\'{e}orie}.
\newblock Ann. Sci. \'{E}cole Norm. Sup. (4), {\bf 1}:161--270, 1968.

\bibitem[Kra15]{Krause-KS-cats-and-projective-covers}
H.~Krause.
\newblock {\em Krull--{S}chmidt categories and projective covers}.
\newblock Expo. Math., {\bf 33}(4):535--549, 2015.

\bibitem[Kru25]{Krull-uber-verallgemeinerte-endliche-abelsche-gruppen}
W.~Krull.
\newblock {\em {\"{U}}ber verallgemeinerte endliche {A}belsche {G}ruppen}.
\newblock Math. Z., {\bf 23}(1):161--196, 1925.

\bibitem[Lam01]{Lam-first-course-noncomm-rings}
T.~Y. Lam.
\newblock {\em A first course in noncommutative rings}, volume 131 of Graduate
  Texts in Mathematics.
\newblock Springer-Verlag, New York, second edition, 2001.

\bibitem[LNP13]{LiuNgPaquette-almost-split-sequences-and-approximations}
S.~Liu, P.~Ng, and C.~Paquette.
\newblock {\em Almost split sequences and approximations}.
\newblock Algebr. Represent. Theory, {\bf 16}(6):1809--1827, 2013.

\bibitem[{Mac}98]{MacLane-categories-for-the-working-mathematician}
S.~{Mac Lane}.
\newblock {\em Categories for the working mathematician}, volume~5 of Grad.
  Texts in Math.
\newblock Springer-Verlag, New York, second edition, 1998.

\bibitem[Mer92]{Merzbach-Robert-Remak-and-the-estimation-of-units-and-regulators}
U.~C. Merzbach.
\newblock {\em Robert {R}emak and the estimation of units and regulators}.
\newblock In Amphora, pages 481--522. Birkh{\"{a}}user, Basel, 1992.

\bibitem[MW09]{Maclagan-Wedderburn-On-the-direct-product-in-the-theory-of-finite-groups}
J.~H. Maclagan-Wedderburn.
\newblock {\em On the direct product in the theory of finite groups}.
\newblock Ann. of Math. (2), {\bf 10}(4):173--176, 1909.

\bibitem[Rem11]{Remak-Uber-die-zerlegung-der-endlichen-gruppen-in-direkte-unzerlegbare-faktoren}
R.~Remak.
\newblock {\em {\"{U}}ber die {Z}erlegung der endlichen {G}ruppen in direkte
  unzerlegbare {F}aktoren}.
\newblock J. Reine Angew. Math., {\bf 139}:293--308, 1911.

\bibitem[Rin84]{Ringel-tame-algebras-and-integral-quadratic-forms}
C.~M. Ringel.
\newblock {\em Tame algebras and integral quadratic forms}, volume 1099 of
  Lecture Notes in Mathematics.
\newblock Springer-Verlag, Berlin, 1984.

\bibitem[Sch13]{Schmidt-Sur-les-produits-directs}
O.~Schmidt.
\newblock {\em Sur les produits directs}.
\newblock Bull. Soc. Math. France, {\bf 41}:161--164, 1913.

\bibitem[Seg03]{Segal-Mathematicians-under-the-nazis}
S.~L. Segal.
\newblock {\em Mathematicians under the {N}azis}.
\newblock Princeton University Press, Princeton, NJ, 2003.

\bibitem[WW76]{WalkerWarfield-Unique-decomposition-and-isomorphic-refinement-theorems-in-additive-categories}
C.~L. Walker and R.~B. Warfield, Jr.
\newblock {\em Unique decomposition and isomorphic refinement theorems in
  additive categories}.
\newblock J. Pure Appl. Algebra, {\bf 7}(3):347--359, 1976.

\end{thebibliography}
\end{document}